\documentclass[11pt,reqno]{amsart}

\usepackage{amssymb,amsmath}
\usepackage{url}		
\numberwithin{equation}{section}
\newtheorem{theorem}[equation]{Theorem}
\newtheorem{lemma}[equation]{Lemma}

\theoremstyle{definition}
\newtheorem{example}[equation]{Example}

\newcommand{\R}{{\mathbb R}}

\newcommand{\Z}{{\mathbb Z}}

\title[Convolutions of Tribonacci, Fuss--Catalan, and Motzkin sequences]{Convolutions of Tribonacci, Fuss--Catalan, \\ and Motzkin sequences}
\author{Daniel Birmajer}
\address{Department of Mathematics\\
                Nazareth College\\ 4245 East Ave. \\
                Rochester, NY 14618, U.S.A.}
\email{abirmaj6@naz.edu}
\author{Juan B. Gil}
\author{Michael D. Weiner}
\address{Penn State Altoona\\
               3000 Ivyside Park\\
               Altoona, PA 16601, U.S.A.}
\email{jgil@psu.edu}
\email{mdw8@psu.edu}

\begin{document}
\maketitle

\begin{abstract}
We introduce a class of sequences, defined by means of partial Bell polynomials, that contains a basis for the space of linear recurrence sequences with constant coefficients as well as other well-known sequences like Catalan and Motzkin. For the family of `Bell sequences' considered in this paper, we give a general multifold convolution formula and illustrate our result with a few explicit examples.  
\end{abstract}

\section{Introduction}
Given numbers $a$ and $b$, not both equal to zero, and given a sequence $c_1, c_2, \dots$, we consider the sequence $(y_n)$ given by
\begin{equation} \label{bell-transform}
 y_0=1, \quad
 y_n = \sum_{k=1}^{n} \binom{a n+b k}{k-1}\frac{(k-1)!}{n!} B_{n,k}(1!c_1,2!c_2,\dots) \text{ for } n\ge 1,
\end{equation}
where $B_{n,k}$ denotes the $(n,k)$-th partial Bell polynomial defined as
\begin{equation*}
  B_{n,k}(x_1,\dots,x_{n-k+1})=\sum_{\alpha\in\pi(n,k)} \frac{n!}{\alpha_1! \alpha_2!\cdots \alpha_{n-k+1}!}\big(\tfrac{x_1}{1!}\big)^{\alpha_1}\cdots \big(\tfrac{x_{n-k+1}}{(n-k+1)!}\big)^{\alpha_{n-k+1}}
\end{equation*}
with $\pi(n,k)$ denoting the set of multi-indices $\alpha\in{\mathbb N}_0^{n-k+1}$ such that $\alpha_1+\alpha_2+\cdots=k$ and $\alpha_1+2\alpha_2+3\alpha_3+\cdots=n$. For more about Bell polynomials, see e.g. \cite[Chapter~3]{Comtet}. In general, there is no need to impose any restriction on the entries $x_1, x_2, \dots$ other than being contained in a commutative ring. Here we are mainly interested in $\Z$ and $\Z[x]$.

The class of sequences \eqref{bell-transform} turns out to offer a unified structure to a wide collection of known sequences. For instance, with $a=0$ and $b=1$, any linear recurrence sequence with constant coefficients $c_1, c_2, \dots,c_d$, can be written as a linear combination of sequences of the form \eqref{bell-transform}. 
In fact, if $(a_n)$ is a recurrence sequence satisfying $a_n=c_1 a_{n-1}+c_2a_{n-2}+\cdots+c_da_{n-d}$ for $n\ge d$, then there are constants $\lambda_0, \lambda_1,\dots, \lambda_{d-1}$ (depending on the initial values of the sequence) such that $a_n= \lambda_0 y_n + \lambda_1 y_{n-1}+\cdots+\lambda_{d-1} y_{n-d+1}$ with
\begin{equation*}
 y_0=1, \quad y_n = \sum_{k=1}^n \frac{k!}{n!} B_{n,k}(1!c_1, 2!c_2, \dots) \text{ for } n\ge 1.
\end{equation*}
For more details about this way of representing linear recurrence sequences, cf. \cite{BGW14}. 

On the other hand, if $a=1$ and $b=0$, we obtain sequences like Catalan and Motzkin by making appropriate choices of $c_1$ and $c_2$, and by setting $c_j=0$ for $j\ge 3$. These and other concrete examples will be discussed in sections \ref{sec:Examples1} and \ref{sec:Examples2}.

In this paper, we focus on convolutions and will use known properties of the partial Bell polynomials to prove a multifold convolution formula for \eqref{bell-transform}.

\section{Convolution Formula}
\label{sec:Convolutions}

Our main result is the following formula.
\begin{theorem}\label{thm:r-convolution-general}
Let $y_0=1$ and for $n\ge 1$,
\begin{equation*}
 y_n = \sum_{k=1}^n \binom{a n+b k}{k-1}\frac{(k-1)!}{n!} B_{n,k}(1!c_1, 2!c_2, \dots).
\end{equation*}
For $r\ge 1$, we have
\begin{equation}\label{eqn:r-convolution}
\sum_{m_1+\dots+m_r=n} \!\! y_{m_1}\cdots y_{m_r} =r\sum_{k=1}^{n}\binom{a n+b k+r-1}{k-1} \frac{(k-1)!}{n!} B_{n,k}(1!c_1, 2!c_2, \dots).
\end{equation}
\end{theorem}

In order to prove this theorem, we recall a convolution formula for partial Bell polynomials that was given by the authors in \cite[Section~3, Corollary~11]{BGW12}.
\begin{lemma}\label{lem:basicConvolution}
Let $\alpha(\ell,m)$ be a linear polynomial in $\ell$ and $m$. For any $\tau\not=0$, we have
\begin{equation*}
\sum_{\ell=0}^{k}\sum_{m=\ell}^n \frac{\binom{\alpha(\ell,m)}{k-\ell}\binom{\tau-\alpha(\ell,m)}{\ell}\binom{n}{m}}{\alpha(\ell,m)(\tau-\alpha(\ell,m))\binom{k}{\ell}} B_{m,\ell} B_{n-m,k-\ell} = \frac{\tau-\alpha(0,0)+\alpha(k,n)}{\tau \alpha(k,n)(\tau-\alpha(0,0))}\binom{\tau}{k} B_{n,k}.
\end{equation*}
\end{lemma}

This formula is key for proving Theorem~\ref{thm:r-convolution-general}. For illustration purposes, we start by proving the special case of a simple convolution (i.e. $r=2$).
\begin{lemma}
The sequence $(y_n)$ defined by \eqref{bell-transform} satisfies 
\[ \sum_{m=0}^{n} y_{m}\,y_{n-m} 
  = 2\sum_{k=1}^{n}\binom{an+b k+1}{k-1} \frac{(k-1)!}{n!} B_{n,k}(1!c_1, 2!c_2, \dots). \]
\end{lemma}
\begin{proof}  We begin by assuming $a, b\ge 0$. For $n\ge0$ we can rewrite $y_n$ as
\begin{equation}\label{yn_rewritten} 
y_n = \sum_{k=0}^{n} \frac1{an+bk + 1}\binom{an+bk+1}{k}\frac{k!}{n!} B_{n,k}(1!c_1, 2!c_2, \dots). 
\end{equation}
By definition,
\begin{align*}
 \sum_{m=0}^{n} & y_{m}\,y_{n-m} \\
 &= \sum_{m=0}^{n} \Bigg[\sum_{\ell=0}^{m} \tfrac1{a m+b \ell+1}\tbinom{a m+ b \ell+1}{\ell}\tfrac{\ell!}{m!} B_{m,\ell}\Bigg]\! \Bigg[\sum_{j=0}^{n-m} \tfrac1{a (n-m)+b j+1}\tbinom{a (n-m)+ b j+1}{j}\tfrac{j!}{(n-m)!} B_{n-m,j}\Bigg] \\
 &= \sum_{m=0}^{n} \sum_{k=0}^{n} \sum_{\ell=0}^{k}\frac{\binom{a m+ b \ell+1}{\ell}\binom{a (n-m)+ b (k-\ell)+1}{k-\ell}}{(a m+b \ell+1)(a (n-m)+b (k-\ell)+1)}\frac{\ell!}{m!} \frac{(k-\ell)!}{(n-m)!} B_{m,\ell}B_{n-m,k-\ell} \\
 &= \sum_{k=0}^{n} \frac{k!}{n!}\sum_{\ell=0}^{k} \sum_{m=\ell}^{n} \frac{\binom{a (n-m)+ b (k-\ell)+1}{k-\ell}\binom{am+b\ell+1}{\ell} \binom{n}{m}}{(am+b\ell+1)(a(n-m)+b(k-\ell)+1) \binom{k}{\ell}} B_{m,\ell}B_{n-m,k-\ell}
\end{align*}
\begin{align*}
 &= \sum_{k=0}^{n} \frac{k!}{n!}\Bigg[ \sum_{\ell=0}^{k} \sum_{m=\ell}^{n} \frac{\binom{\alpha(\ell,m)}{k-\ell}\binom{\tau - \alpha(\ell,m)}{\ell}\binom{n}{m}}{(\tau-\alpha(\ell,m))\alpha(\ell,m)\binom{k}{\ell}} B_{m,\ell}B_{n-m,k-\ell}\Bigg]
\end{align*}
with $\alpha(\ell,m)=a(n-m)+b(k-\ell)+1$ and $\tau=an + bk +2$. 

Thus, by Lemma~\ref{lem:basicConvolution},
\begin{align*}
 \sum_{m=0}^{n} y_{m}\,y_{n-m} 
 &= \sum_{k=0}^{n} \frac{k!}{n!}\left[\frac{\tau-\alpha(0,0)+\alpha(k,n)}{\tau \alpha(k,n)(\tau-\alpha(0,0))}\binom{\tau}{k} B_{n,k}(1!c_1, 2!c_2, \dots)\right] \\
 &= \sum_{k=0}^{n} \frac{k!}{n!}\left[\frac{2}{(an+bk+2)}\binom{an+bk +2}{k} B_{n,k}(1!c_1, 2!c_2, \dots)\right] \\
&=2\sum_{k=0}^{n}\binom{an+bk+1}{k-1} \frac{(k-1)!}{n!} B_{n,k}(1!c_1, 2!c_2, \dots).
\end{align*}
For any fixed $n$, both sides of the claimed equation are polynomials in $a$ and $b$. Since they coinside on an open subset of $\R^2$, they must coincide for all real numbers $a$ and $b$.
\end{proof}

\begin{proof}[\bf Proof of Theorem~\ref{thm:r-convolution-general}]
We proceed by induction in $r$. The case $r=2$ was discussed in the previous lemma. Assume the formula \eqref{eqn:r-convolution} holds for products of length less than $r>2$.  

As before, we temporarily assume that both $a$ and $b$ are positive. For $n\ge0$ we rewrite
\[ \sum_{m_1+\dots+m_{r-1}=n} \!\! y_{m_1}\cdots y_{m_{r-1}} =\sum_{k=0}^{n}\frac{r-1}{a n+b k+r-1}\tbinom{a n+b k+r-1}{k} \frac{k!}{n!} B_{n,k}(1!c_1, 2!c_2, \dots). \]
Thus
\begin{align*}
 \sum_{m_1+\dots+m_r=n} \!\! y_{m_1}\cdots y_{m_r} 
 &= \sum_{m=0}^n y_m \!\! \sum_{m_1+\dots+m_{r-1}=n-m} \!\! y_{m_1}\cdots y_{m_{r-1}} \\
 &= \sum_{m=0}^n y_m \sum_{j=0}^{n-m}\tfrac{r-1}{a(n-m)+bj+r-1}\tbinom{a(n-m)+bj+r-1}{j} \tfrac{j!}{(n-m)!} B_{n-m,j}.
\end{align*}
Writing $y_m$ as in \eqref{yn_rewritten}, we then get
\begin{align*}
  \frac{1}{r-1} &\sum_{m_1+\dots+m_r=n} \!\! y_{m_1}\cdots y_{m_r} \\
  &= \sum_{m=0}^n \Bigg[\sum_{\ell=0}^{m} \frac{\tbinom{am+b\ell+1}{\ell} \ell!}{(am+b\ell + 1)m!} B_{m,\ell}\Bigg] \! \Bigg[\sum_{j=0}^{n-m}\frac{\tbinom{a(n-m)+bj+r-1}{j}j!}{(a(n-m)+bj+r-1)(n-m)!} B_{n-m,j}\Bigg] \\ 
  &= \sum_{m=0}^n \sum_{k=0}^n \sum_{\ell=0}^{k} \frac{\binom{a (n-m)+b (k-\ell)+r-1)}{k-\ell}\binom{a m+b \ell+1}{\ell}}{(a m+b \ell + 1)(a (n-m)+b (k-\ell)+r-1)}\frac{\ell!}{m!} \frac{(k-\ell)!}{(n-m)!} B_{m,\ell}B_{n-m,k-\ell}\\
 &= \sum_{k=0}^{n} \frac{k!}{n!}\Bigg[ \sum_{\ell=0}^{k} \sum_{m=\ell}^{n} \frac{\binom{\alpha(\ell,m)}{k-\ell}\binom{\tau - \alpha(\ell,m)}{\ell}\binom{n}{m}}{(\tau-\alpha(\ell,m))\alpha(\ell,m)\binom{k}{\ell}} B_{m,\ell}B_{n-m,k-\ell}\Bigg]
\end{align*}
with $\alpha(\ell,m)=a(n-m)+b(k-\ell)+r-1$ and $\tau=an + bk +r$. 

Finally, by Lemma~\ref{lem:basicConvolution},
\begin{align*}
 \sum_{m_1+\dots+m_r=n} \!\! y_{m_1}\cdots y_{m_r}
 &= (r-1)\sum_{k=0}^{n} \frac{k!}{n!}\bigg[\frac{\tau-\alpha(0,0)+\alpha(k,n)}{\tau \alpha(k,n)(\tau-\alpha(0,0))}\binom{\tau}{k} B_{n,k}(1!c_1, 2!c_2, \dots)\bigg] \\
 &= (r-1)\sum_{k=0}^{n} \frac{k!}{n!}\bigg[\frac{r\tbinom{an+bk +r}{k}}{(an+bk+r)(r-1)} B_{n,k}(1!c_1, 2!c_2, \dots)\bigg] \\
&=r\sum_{k=0}^{n}\binom{an+bk+r-1}{k-1} \frac{(k-1)!}{n!} B_{n,k}(1!c_1, 2!c_2, \dots).
\end{align*}
As in the previous lemma, this equation actually holds for all $a,b\in\R$ as claimed.
\end{proof}

\section{Examples: Fibonacci, Tribonacci, Jacobsthal}
\label{sec:Examples1}

As mentioned in the introduction, sequences of the form \eqref{bell-transform} with $a=0$ and $b=1$ can be used to describe linear recurrence sequences with constant coefficients.  In this case, \eqref{bell-transform} takes the form
\begin{equation} \label{bell-transform01}
  y_n = \sum_{k=0}^n \frac{k!}{n!} B_{n,k}(1!c_1, 2!c_2, \dots) \text{ for } n\ge 0,
\end{equation}
and the convolution formula \eqref{eqn:r-convolution} turns into
\begin{align*}
\sum_{m_1+\dots+m_r=n} \!\! y_{m_1}\cdots y_{m_r} 
 &= r\sum_{k=1}^{n}\binom{k+r-1}{k-1} \frac{(k-1)!}{n!} B_{n,k}(1!c_1, 2!c_2, \dots) \\
 &= \sum_{k=1}^{n}\binom{k+r-1}{k} \frac{k!}{n!} B_{n,k}(1!c_1, 2!c_2, \dots).
\end{align*}
One can obtain (with a similar proof) the more general formula
\begin{equation*}
\sum_{m_1+\dots+m_r=n} \!\! y_{m_1-\delta}\cdots y_{m_r-\delta}
  =\sum_{k=0}^{n-\delta r} \binom{k+r-1}{k}\frac{k!}{(n-\delta r)!} B_{n-\delta r,k}(1!c_1, 2!c_2, \dots)
\end{equation*}
for any integer $\delta\ge 0$, assuming $y_{-1}=y_{-2}=\cdots=y_{-\delta}=0$.

\begin{example}[Fibonacci]
Consider the sequence defined by
\begin{gather*}
  f_{0}=0, \;\; f_{1}=1, \text{ and }\; f_{n}=f_{n-1}+f_{n-2} \;\text{ for } n\ge 2.
\end{gather*}
Choosing $c_1=c_2=1$ and $c_j=0$ for $j\ge 3$ in \eqref{bell-transform01}, for $n\ge 1$ we have 
\begin{equation*}
  f_{n} = y_{n-1} = \sum_{k=0}^{n-1} \frac{k!}{(n-1)!} B_{n-1,k}(1, 2, 0, \dotsc)= \sum_{k=0}^{n-1} \binom{k}{n-1-k},
\end{equation*}
and
\begin{equation*}
\sum_{m_1+\dots+m_r=n} \!\! f_{m_1}\cdots f_{m_r}= \sum_{k=0}^{n- r} \binom{k+r-1}{k}\binom{k}{n-r-k}.
\end{equation*}
\end{example}
\begin{example}[Tribonacci]
Let $(t_n)$ be the sequence defined by
\begin{equation*}
  t_0=t_1=0,\;\; t_2=1, \text{ and }\; t_n=t_{n-1}+t_{n-2}+t_{n-3} \text{ for } n\ge 3.
\end{equation*}
Choosing $c_1=c_2=c_3=1$ and $c_j=0$ for $j\ge 4$ in \eqref{bell-transform01}, for $n\ge 2$ we have 
\begin{equation*}
 t_n= y_{n-2} = \sum_{k=0}^{n-2} \frac{k!}{(n-2)!} B_{n-2,k}(1!,2!,3!,0,\dots),
\end{equation*}
and since $B_{n,k}(1!,2!,3!,0,\dots)=\frac{n!}{k!}\sum_{\ell=0}^k \binom{k}{k-\ell} \binom{k-\ell}{n+\ell-2k} = \frac{n!}{k!}\sum_{\ell=0}^k \binom{k}{\ell} \binom{\ell}{n-k-\ell}$, we get
\begin{equation*}
 t_n =\sum_{k=0}^{n-2} \sum_{\ell=0}^k \binom{k}{\ell} \binom{\ell}{n-2-k-\ell},
\end{equation*}
and
\begin{equation*}
\sum_{m_1+\dots+m_r=n} \!\! t_{m_1}\cdots t_{m_r}
  =\sum_{k=0}^{n-2r}\sum_{\ell=0}^k \binom{k+r-1}{k} \binom{k}{\ell} \binom{\ell}{n-2r-k-\ell}.
\end{equation*}
\end{example}
\begin{example}[Jacobsthal]
The Jacobsthal polynomials are obtained by the recurrence 
\begin{gather*}
 J_0=0, \;\; J_1=1, \text{ and} \\
 J_n=J_{n-1}+2x J_{n-2} \;\text{ for } n\ge 2.
\end{gather*}
Choosing $c_1=1$, $c_2=2x$, and $c_j=0$ for $j\ge 3$ in \eqref{bell-transform01}, for $n\ge 1$ we get
\begin{equation*}
J_n = y_{n-1} = \sum_{k=0}^{n-1} \frac{k!}{(n-1)!} B_{n-1,k}(1,2(2x), 0, \dotsc)
= \sum_{k=0}^{n-1} \binom{k}{n-1-k}(2x)^{n-1-k},
\end{equation*}
and
\begin{align*}
\sum_{m_1+\dots+m_r=n} \!\! J_{m_1}\cdots J_{m_r} 
 &=\sum_{k=0}^{n-r} \frac{k!}{(n-r)!} \binom{k+r-1}{k}B_{n-r,k}(1,4x,0,\dots)\\
 &=\sum_{k=0}^{n-r} \binom{k+r-1}{k} \binom{k}{n-r-k} (2x)^{n-r-k}.
\end{align*}
\end{example}

\section{Examples: Fuss--Catalan, Motzkin}
\label{sec:Examples2}

All of the previous examples are related to the family \eqref{bell-transform01}. However, there are many other cases of interest. For example, let us consider the case when $a=1$, $b=0$, and $c_j=0$ for $j\ge 3$. Since $B_{n,k}(c_1,2c_2,0,\dots)=\frac{n!}{k!} \binom{k}{n-k} c_1^{2k-n}c_2^{n-k}$, the family \eqref{bell-transform} can be written as
\begin{equation} \label{bell-transform10}
 y_0=1, \quad
 y_n = \sum_{k=1}^{n} \frac1k \binom{n}{k-1} \binom{k}{n-k} c_1^{2k-n}c_2^{n-k} \text{ for } n\ge 1,
\end{equation}
and the convolution formula \eqref{eqn:r-convolution} becomes
\begin{equation} \label{eq:convolution10}
\sum_{m_1+\dots+m_r=n} \!\! y_{m_1}\cdots y_{m_r} =\sum_{k=1}^{n} \frac{r}{k} \binom{n+r-1}{k-1} \binom{k}{n-k} c_1^{2k-n}c_2^{n-k}.
\end{equation}
\begin{example}[Catalan] 
If we let $c_1=2$ and $c_2=1$ in \eqref{bell-transform10}, for $n\ge 1$ we get
\begin{align*}
 y_n &= \sum_{k=1}^{n} \frac1k \binom{n}{k-1} \binom{k}{n-k} 2^{2k-n} \\
 &= \frac{1}{n+1}\sum_{k=1}^{n} \binom{n+1}{k} \binom{k}{n-k} 2^{2k-n} \\
 &= \frac{1}{n+1} \binom{2(n+1)}{n} = \frac{1}{n+2} \binom{2(n+1)}{n+1} = C_{n+1}.
\end{align*}
Here we used the identity 
\begin{equation} \label{eq:Gould3.22}
 \sum_{k=\lceil\frac{n}{2}\rceil}^{n} \binom{x}{k} \binom{k}{n-k} 2^{2k} = 2^n\binom{2x}{n} 
\end{equation}
from Gould's collection \cite[Identity (3.22)]{Gould}.  As for convolutions, \eqref{eq:convolution10} leads to 
\begin{align*}
\sum_{m_1+\dots+m_r=n} \!\! C_{m_1+1}\cdots C_{m_r+1} 
 &=\sum_{k=1}^{n} \frac{r}{k} \binom{n+r-1}{k-1} \binom{k}{n-k} 2^{2k-n} \\
 &=\frac{r}{n+r}\sum_{k=1}^{n} \binom{n+r}{k} \binom{k}{n-k} 2^{2k-n}.
\end{align*}
Using again \eqref{eq:Gould3.22}, we arrive at the identity
\begin{equation*}
\sum_{m_1+\dots+m_r=n} \!\! C_{m_1+1}\cdots C_{m_r+1} = \frac{r}{n+r} \binom{2(n+r)}{n}.
\end{equation*}
\end{example}
\begin{example}[Motzkin]
Let us now consider \eqref{bell-transform10} with $c_1=1$ and $c_2=1$. For $n\ge 1$,
\begin{equation*}  
 y_n = \sum_{k=1}^{n} \frac1k \binom{n}{k-1} \binom{k}{n-k} 
 = \frac{1}{n+1}\sum_{k=1}^{n} \binom{n+1}{k} \binom{k}{n-k}.
\end{equation*}
These are the Motzkin numbers $M_n$. Moreover,
\begin{equation*}
\sum_{m_1+\dots+m_r=n} \!\! M_{m_1}\cdots M_{m_r} =\frac{r}{n+r}\sum_{k=0}^{n} \binom{n+r}{k} \binom{k}{n-k}.
\end{equation*}
\end{example}

\medskip
We finish this section by considering the sequence (with $b\not=0$):
\begin{equation*}
 y_0=1, \quad
 y_n = \sum_{k=1}^n \binom{b k}{k-1}\frac{(k-1)!}{n!} B_{n,k}(1!c_1, 2!c_2, \dots) \text{ for } n\ge 1.
\end{equation*}
\begin{example}[Fuss--Catalan]
If $c_1=1$ and $c_j=0$ for $j\ge 2$, then the above sequence becomes
\begin{equation*}
 y_0=1, \quad y_n = \binom{b n}{n-1}\frac{(n-1)!}{n!} = \frac{1}{(b-1)n+1} \binom{b n}{n}.
\end{equation*}
Denoting $C^{(b)}_n=y_n$, and since $r\binom{bn+r-1}{n-1} \frac{(n-1)!}{n!}=\frac{r}{bn+r}\binom{bn+r}{n}$, we get the identity
\begin{equation*}
\sum_{m_1+\dots+m_r=n} \!\! C^{(b)}_{m_1}\cdots C^{(b)}_{m_r} =\frac{r}{bn+r}\binom{bn+r}{n}. 
\end{equation*}
\end{example}


\medskip

\noindent MSC2010: 11B37, 05A19

\end{document}